\documentclass[11pt]{article}
\usepackage{amsmath,amsfonts,amssymb,amsthm}
\usepackage{mathrsfs,amscd,comment,scalefnt}
\usepackage{makeidx}  
\usepackage{graphicx}
\usepackage[matrix,arrow,curve]{xy}
\usepackage{mathabx}

\usepackage{upgreek,relsize} 

\usepackage[T2A]{fontenc}
\usepackage[utf8]{inputenc}
\usepackage[english]{babel}

\usepackage{bm}
\usepackage{mathrsfs}

\usepackage{xcolor}

\newcommand\ScaleExists[1]{\vcenter{\hbox{\scalefont{#1}$\exists$}}}

\newcommand\ScaleForall[1]{\vcenter{\hbox{\scalefont{#1}$\forall$}}}

\DeclareMathOperator*\bigexists{%
  \vphantom\sum
  \mathchoice{\ScaleExists{1.8}}{\ScaleExists{1.4}}{\ScaleExists{1.0}}{\ScaleExists{0.8}}}

\DeclareMathOperator*\bigforall{%
  \vphantom\sum
  \mathchoice{\ScaleForall{1.8}}{\ScaleForall{1.4}}{\ScaleForall{1.0}}{\ScaleForall{0.8}}}

\makeatletter
\def\blfootnote{\xdef\@thefnmark{}\@footnotetext}
\makeatother

\usepackage{hyperref}

\newcommand\hide[1]{\commented{gray}{Hidden:}{#1}}
\renewcommand\hide[1]\empty

\definecolor{teal}{rgb}{0.0, 0.5, 0.5}

\newcommand\DS[1]{{\color{blue}#1}}

\theoremstyle{definition}

\begin{document}

\theoremstyle{plain}
\newtheorem{tm}{Theorem}
\newtheorem{lm}{Lemma}
\newtheorem{prp}{Proposition}
\newtheorem{cor}{Corollary}

\theoremstyle{definition}
\newtheorem{df}{Definition}
\newtheorem{exm}{Example}
\newtheorem*{mdl}{Model} 
\newtheorem{prb}{Problem}

\theoremstyle{remark}
\newtheorem*{pf}{Proof}
\newtheorem{rmk}{Remark}
\newtheorem*{hrmk}{Historical remarks}
\newtheorem*{wrn}{Warning}
\newtheorem{q}{Question}
\newtheorem*{qs}{Questions}
\newtheorem{ex}{Example}
\newtheorem*{exs}{Examples}
\newtheorem*{conj}{Conjecture}

\newcommand{\dom}{ {\mathop{\mathrm {dom\,}}\nolimits} }
\newcommand{\ran}{ {\mathop{\mathrm{ran\,}}\nolimits} }

\newcommand{\id}{ {\mathop{\mathrm {id}}\nolimits} }
\newcommand{\nat}{ {\mathop{\mathrm {nat}}\nolimits} }

\newcommand{\cf}{ {\mathop{\mathrm {cf\,}}\nolimits} }
\newcommand{\cl}{ {\mathop{\mathrm {cl\,}}\nolimits} }
\newcommand{\cof}{ {\mathop{\mathrm {cof\,}}\nolimits} }
\newcommand{\add}{ {\mathop{\mathrm {add\,}}\nolimits} }
\newcommand{\sat}{ {\mathop{\mathrm {sat\,}}\nolimits} }
\newcommand{\tc}{ {\mathop{\mathrm {tc\,}}\nolimits} }
\newcommand{\unif}{ {\mathop{\mathrm {unif\,}}\nolimits} }
\newcommand{\fp}{ {\mathop{\mathrm {fp\,}}\nolimits} }
\newcommand{\fs}{ {\mathop{\mathrm {fs\,}}\nolimits} }
\newcommand{\pr}{ {\mathop{\mathrm {pr\/}}\nolimits} }
\newcommand{\uhr}{\!\upharpoonright\!}
\newcommand{\lra}{ {\:\leftrightarrow\:} }
\newcommand{\ot}{ {\mathop{\mathrm {ot\,}}\nolimits} }
\newcommand{\ol}{\overline}
\newcommand{\cnc}{ {^\frown} }
\newcommand{\image}{\/``\,}

\newcommand{\wt}{\widetilde}
\newcommand{\fnc}{\mathrm{fnc}}

\newcommand{\isom}{\simeq}

\newcommand{\gen}{\textrm{-}\rm{gen}}

\renewcommand{\mathcal}{\mathscr}

\renewcommand{\bm}{\boldsymbol}

\newcommand{\scc}{\bm\upbeta}

\hide{
\renewcommand{\vartheta}{\bm\vartheta}
}

\newcommand{\J}{ {\mathrm J} }
\newcommand{\PC}{ {\mathrm {PC}} }
\newcommand{\PRA}{ {\mathrm {PRA}} }
\newcommand{\T}{ {\mathrm T} }
\newcommand{\TA}{ {\mathrm {TA}} }
\newcommand{\PA}{ {\mathrm {PA}} }
\newcommand{\KP}{ {\mathrm {KP}} }
\newcommand{\Z}{ {\mathrm Z} }
\newcommand{\ZF}{ {\mathrm {ZF}} }
\newcommand{\ZFA}{ {\mathrm {ZFA}} }
\newcommand{\ZFC}{ {\mathrm {ZFC}} }
\newcommand{\AEx}{ {\mathrm {AE}} }
\newcommand{\AR}{ {\mathrm {AR}} }
\newcommand{\WR}{ {\mathrm {WR}} }
\newcommand{\AF}{ {\mathrm {AF}} }
\newcommand{\AC}{ {\mathrm {AC}} }
\newcommand{\GC}{ {\mathrm {GC}} }
\newcommand{\DC}{ {\mathrm {DC}} }
\newcommand{\AD}{ {\mathrm {AD}} }
\newcommand{\AInf}{ {\mathrm {AInf}} }
\newcommand{\AU}{ {\mathrm {AU}} }
\newcommand{\AP}{ {\mathrm {AP}} }
\newcommand{\PI}{ {\mathrm {PI}} }
\newcommand{\CH}{ {\mathrm {CH}} }
\newcommand{\GCH}{ {\mathrm {GCH}} }
\newcommand{\APr}{ {\mathrm {APr}} }
\newcommand{\ASp}{ {\mathrm {ASp}} }
\newcommand{\ARp}{ {\mathrm {ARp}} }
\newcommand{\AFA}{ {\mathrm {AFA}} }
\newcommand{\BAFA}{ {\mathrm {BAFA}} }
\newcommand{\FAFA}{ {\mathrm {FAFA}} }
\newcommand{\SAFA}{ {\mathrm {SAFA}} }
\newcommand{\nonempty}{ {\mathrm {nonempty}} }
\newcommand{\I}{ {\mathrm {I}} }
\newcommand{\II}{ {\mathrm {II}} }
\newcommand{\lex}{ {\mathrm {lex}} }

\newcommand{\fin}{\mathrm{fin}}

\newcommand{\At}{\mathrm{At}}
\newcommand{\Sbf}{\mathrm{Sbf}}

\author{Nikolai L.~Poliakov, Denis I.~Saveliev}
\title{More on expressibility of satisfiability 
in submodels and extensions}

\date{%30.04.2026
\blfootnote{
{\it MSC~2020}:
Primary 
03B45 % Modal logic
03C75 % Other infinitary logic 
Secondary 
% 03B10 % Higher-order logic 
03C52 % Properties of classes of models 
% 03C80 % Logic with extra quantifiers & operations 
03C85 % Second- & higher-order model theory 
03E10 % Ordinal & cardinal numbers 
03E55 % Large cardinals 
}
 \blfootnote{
{\it Keywords}: 
submodel modality, 
extension modality, 
infinitary language, 
monadic language, 
large cardinal
}
%\thanks{The research is supported by the MSHE project}
}

\maketitle
%\newpage 

\begin{abstract}
We study expressibility in infinitary languages 
of the modal operators associated with satisfiability 
of sentences of these languages in submodels and 
extensions of models. We give a syntactic criterion 
for expressibility in finitary predicate languages, 
show that in many cases infinitary languages are 
closed under the operator associated with submodels, 
and that this is so in any language with 
a~purely monadic signature. Finally, we prove that 
in finitary or strongly compact languages, 
the operator associated with extensions, though 
can be inexpressible by a single sentence, 
is always expressible by a universal theory, 
in striking contrast with the submodel case. 
\end{abstract}

%\newpage 

\section*{Introduction}

\hide{
\DS{ perhaps better notation: 
$\bm\vartheta(\varphi)$ and $\bm\vartheta^*(\varphi)$ 
or else
$\bm\upvartheta(\varphi)$ and $\bm\upvartheta^*(\varphi)$ }
}

\vspace{0.5em}

For a~class~$K$ of models of a~signature~$\tau$, 
a~binary relation~$R$ on~$K$, and a~sentence~$\varphi$ 
of a~model-theoretic language~$\mathcal L$ 
(in sense of~\cite{Barwise Feferman}), let 
$\vartheta_{K,R}(\varphi)$ express the fact that 
for any model $\mathfrak A$ in~$K$ there exists 
a~model $\mathfrak B$ in~$K$ such that 
$\mathfrak A\,R\,\mathfrak B$ and 
$\mathfrak B\vDash\varphi$. Such $\vartheta_{K,R}$, 
treated as modal operators, lead to modal logics, 
which were, in the general form, introduced and 
studied in~\cite{Saveliev Shapirovsky 2020} 
(for previously known special cases, see 
\cite{Saveliev Shapirovsky 2020}, Introduction). 
One of problems arising in the study of these logics 
is to determine languages~$\mathcal L'$ to which 
the resulting $\vartheta_{K,R}(\varphi)$ belongs 
for a~given $\varphi$ in~$\mathcal L$, in particular, 
to determine when it belongs to $\mathcal L$ itself.

In the case when $R$~is the submodel relation, 
i.e., when $\mathfrak A\,R\,\mathfrak B$ means 
that $\mathfrak B$ is a~submodel of~$\mathfrak A$, 
$K$~is the class of all models of a~given 
signature~$\tau$, and $\mathcal L$~is an infinitary 
language $\mathcal L_{\kappa,\lambda}(\tau)$ 
or its second-order counterpart 
$\mathcal L^{2}_{\kappa,\lambda}(\tau)$, this 
problem was investigated in~\cite{Saveliev 2019}
(for resulting modal logics, see 
\cite{Saveliev Shapirovsky 2020}, Section~3,  
for more completeness results Section~4, and 
\cite{Saveliev Shapirovsky 2022}, Section~6). 
Following the notation of~\cite{Saveliev 2019}, 
in this case we write, instead $\vartheta_{K,R}$, 
just~$\vartheta$. Thus for any model~$\mathfrak A$ 
of~$\tau$,
$$
\mathfrak A\vDash\vartheta(\varphi)
\;\text{ iff }\;
\text{there exists a~submodel~$\mathfrak B$ of 
$\mathfrak A$ such that }
\mathfrak B\vDash\varphi.
$$ 
Recall that $\mathcal L_{\omega,\omega}$ is 
the usual first-order finitary language; 
$\mathcal L_{\kappa,\lambda}$ expands 
it by involving Boolean connectives of 
any arities~$<\kappa$ and quantifiers of 
any arities~$<\lambda$, where
$\omega\le\lambda\le\kappa$ and $\kappa$ is 
assumed to be regular (if $\kappa$ is singular, 
$\mathcal L_{\kappa,\lambda}$ is semantically 
equivalent to $\mathcal L_{\kappa^+,\lambda}$, 
see \cite{Dickmann 1975}, Example~1.2.22); 
and $\mathcal L_{\infty,\lambda}$~is the union 
of $\mathcal L_{\kappa,\lambda}$ for all~$\kappa$;
see \cite{Barwise Feferman,
Dickmann 1975,Drake,Kanamori}.

The following facts were stated 
in~\cite{Saveliev 2019}. 
For all $\lambda$, $\kappa$, and $\varphi$ 
in $\mathcal L_{\kappa,\lambda}(\tau)$, 
$\vartheta(\varphi)$ is expressible by 
a~sentence in the second-order language 
$\mathcal L^{2}_{\kappa,\lambda}(\tau)$ iff 
the signature~$\tau$ contains~$<\kappa$ function 
(including constant) symbols 
(\cite{Saveliev 2019}, Theorems 5 and~7). 
Certain first-order languages 
$\mathcal L_{\kappa,\lambda}(\tau)$ are closed 
under~$\vartheta$, in the sense that, 
whenever $\varphi$ is in the language then so 
is $\vartheta(\varphi)$; in particular, this is 
the case if $\kappa$ is strongly inaccessible
(\cite{Saveliev 2019}, Remark~9). However, 
$\mathcal L_{\omega,\omega}(\tau)$ does not 
share this feature; moreover, for $\tau$ with 
a~binary predicate symbol, some $\varphi$ have 
$\vartheta(\varphi)$ inexpressible even in 
$\mathcal L_{\infty,\omega}(\tau)$ 
(\cite{Saveliev 2019}, Theorem~4). 
The following semantic criterion 
specifying sentences $\varphi$ such that 
$\vartheta(\varphi)$ is expressible in 
$\mathcal L_{\omega,\omega}(\tau)$ was proved
(\cite{Saveliev 2019}, Corollary~11):
If $\tau$ is a~signature without function symbols, 
then the following are equivalent:
\begin{quote}
\begin{itemize}
\setlength\itemsep{-0.2em}
\item[(i)]
$\vartheta(\varphi)$ is equivalent to 
a~sentence in $\mathcal L_{\omega,\omega}(\tau)$;
\item[(ii)]
any model satisfying $\varphi$ has a~finite 
submodel satisfying~$\varphi$; 
\item[(iii)]
there exists $n<\omega$ such that 
any model satisfying~$\varphi$ has a~submodel 
of cardinality~$\le n$ satisfying~$\varphi$. 
\end{itemize}
\end{quote}

%\newpage 

This paper continues these studies.

In the first section, we consider the 
finitary language $\mathcal L_{\omega,\omega}$. 
Supplementing the above semantic criterion, 
we provide a~syntactic criterion 
for those sentences $\varphi$ in 
$\mathcal L_{\omega,\omega}(\tau)$ that have 
$\vartheta(\varphi)$ expressible in the same 
$\mathcal L_{\omega,\omega}(\tau)$
(Theorem~\ref{t: syntactic criterion}). 
It says $\varphi$~has this property iff
it is equivalent to a~$\Sigma^{0}_2$-sentence 
in a~certain universal theory~$T_\varphi$. 
We also show that $\vartheta(\varphi)$ can be 
expressible by a~theory but not by a~single 
sentence of the same language, and this 
never happens with $\neg\,\vartheta(\varphi)$
(Propositions \ref{p: theta expressed by thy} 
and~\ref{p: neg-theta inexpressed by thy}).

In the next section, we consider languages 
$\mathcal L_{\kappa,\lambda}$ with $\kappa>\omega$.
First, we write out infinitary formulas expressing 
statements like $\vartheta_{\le\nu}(\varphi)$, 
variants of $\vartheta(\varphi)$ with bounded 
cardinality of submodels, and provide sufficient 
conditions for $\mathcal L_{\kappa,\lambda}(\tau)$
to be closed under~$\vartheta_{\le\nu}$ 
(Lemma~\ref{l: relativization} and 
Corollary~\ref{c: expressibility of theta_nu}).
Then we use infinitary versions of the downward 
L{\"o}wenheim--Skolem theorems to isolate cases 
when $\vartheta_{\le\nu}$ coincides 
with~$\vartheta$ (Theorem~\ref{t: dLS vs theta} 
and Corollary~\ref{c: dLS vs theta}).
Finally, we combine both observations: given 
$\kappa,\lambda$, and~$\tau$, we calculate 
$\kappa',\lambda'$ such that, for any $\varphi$ 
in $\mathcal L_{\kappa,\lambda}(\tau)$, 
$\vartheta(\varphi)$ is expressible 
in $\mathcal L_{\kappa',\lambda'}(\tau)$; 
thence we obtain conditions sufficient for 
$\mathcal L_{\kappa,\lambda}(\tau)$ to be 
closed under~$\vartheta$
(Theorem~\ref{t: expressing theta}). We note 
that (due to Remark~\ref{r: properly many kappa})
for any given $\lambda$ and $\tau$ there is 
a~proper class of~$\kappa$ such that 
$\mathcal L_{\kappa,\lambda}(\tau)$ 
is closed under~$\vartheta$.

In the third section, we first prove that, 
for any $\mathcal L_{\kappa,\lambda}(\tau)$, 
the sentences constructed from prenex 
existential-universal sentences by conjunctions 
and disjunctions, whenever are satisfied in 
a~model, then satisfied in submodels of 
cardinality~$<\kappa$, and conclude that 
$\vartheta$ and an its restricted version 
($\vartheta_{<\kappa}$ for $\tau$~without 
function symbols) coincide on these sentences 
(Proposition~\ref{p: EA} and 
Corollary~\ref{c: EA}). 
We observe then that so-called monadic-like 
sentences have this form; moreover, any such 
sentence of $\mathcal L_{\kappa,\lambda}(\tau)$ 
is equivalent to a~Boolean combination of 
existential (or universal) one-variable 
sentences of $\mathcal L_{\kappa,\omega}(\tau)$ 
(Theorem~\ref{t: monadic-like}). Applying 
these observations to purely monadic (i.e., 
consisting only of unary predicate symbols) 
signatures~$\tau$ without equality, we prove 
that in this case 
$\mathcal L_{\kappa,\lambda}(\tau)$ 
is semantically equivalent to 
$\mathcal L_{\kappa,\omega}(\tau)$ 
and is closed under~$\vartheta$ 
(Theorem~\ref{t: monadic signature}).

In the final section, we consider 
$\vartheta_{K,R}$ in the converse case when 
$R$~is the extension relation, i.e., when 
$\mathfrak A\,R\,\mathfrak B$ means that 
$\mathfrak A$ is a~submodel of~$\mathfrak B$ 
and $K$~is the class of models of 
a~given~$\tau$. To keep the notation short, 
we shall denote it just by~$\vartheta^*$. 
Thus for any model~$\mathfrak A$ of~$\tau$, 
$$
\mathfrak A\vDash\vartheta^*(\varphi)
\;\text{ iff }\;
\text{there is a~model }
\mathfrak B\vDash\varphi 
\text{ such that $\mathfrak A$ is 
a~submodel of $\mathfrak B$.} 
$$
We observe first that 
if $\vartheta^*(\varphi)$ is expressible 
in $\mathcal L_{\omega,\omega}(\tau)$, or 
in $\mathcal L_{\kappa,\lambda}(\tau)$ 
with a~strongly compact~$\kappa$, then 
it is universal 
(Proposition~\ref{p: theta* universal}), 
and that there are sentences $\varphi$ of 
$\mathcal L_{\omega,\omega}(\tau)$ for which 
$\vartheta^*(\varphi)$ is inexpressible by 
a~single sentence of the same language 
% $\mathcal L_{\omega,\omega}(\tau)$ 
(Theorem~\ref{t: inexpressible theta*}).
On the other hand, we state 
the following expressibility result
(Theorem~\ref{t: expressing theta*}). 
For any $\tau$ and $\varphi$ in 
$\mathcal L_{\omega,\omega}(\tau)$, 
or in $\mathcal L_{\kappa,\lambda}(\tau)$ 
with a~strongly compact~$\kappa$, 
$\vartheta^*(\varphi)$ is expressed by 
a~universal theory in the same language, 
and so, by a~single universal sentence in 
$\mathcal L_{\infty,\omega}(\tau)$, 
respectively, 
$\mathcal L_{\infty,\lambda}(\tau)$\,---\,%  
in striking contrast to the case of submodels 
where $\vartheta(\varphi)$ can be inexpressible 
in $\mathcal L_{\infty,\omega}(\tau)$. 
So $\vartheta^*(\varphi)$ is expressible 
by a~single sentence iff $\varphi$ has 
a~strongest universal consequence. Moreover, 
this remains true for the second-order 
languages $\mathcal L^{2}_{\kappa,\lambda}$ 
with an extendible~$\kappa$ 
(Remark~\ref{r: extendible}).

Most of the notation of this paper is 
common in set theory and model theory.  
Given a~signature~$\tau$, we denote by 
$\fnc(\tau)$ its subsignature consisting of 
all function (including constant) symbols. 
We let, for any model~$\mathfrak A$ 
and cardinal $\kappa\ge1$,
\begin{align*}
\mathfrak A\vDash\vartheta_{\kappa}(\varphi)
\;\text{ iff }\;
\text{there exists a~submodel $\mathfrak B$ of 
$\mathfrak A$ such that }|B|=\kappa\text{ and }
\mathfrak B\vDash\varphi;
\\
\mathfrak A\vDash\vartheta_{\kappa\gen}(\varphi)
\;\text{ iff }\;
\text{there exists a~$\kappa$-generated submodel 
$\mathfrak B$ of $\mathfrak A$ such that }
\mathfrak B\vDash\varphi.
\end{align*}
Also, let 
$
\vartheta_{<\nu}(\varphi)
\text{ iff }
\bigvee_{\!\mu<\nu}\vartheta_{\mu}(\varphi), 
$
$
\vartheta_{\le\nu}(\varphi)
\text{ iff }
\bigvee_{\!\mu\le\nu}\vartheta_{\mu}(\varphi), 
$
and likewise for $\vartheta_{<\nu\gen}(\varphi)$ 
and $\vartheta_{\le\nu\gen}(\varphi)$. 
In the finitary case, the class $\Sigma^{0}_1$ 
consists of existential, and $\Pi^{0}_1$, 
of universal formulas, $\Sigma^{0}_2$~of 
existential-universal ones, 
etc.~(see~\cite{Chang Keisler}). In the case of 
infinitary languages, formulas are generally not 
reduced to the prenex form (see \cite{Dickmann 1975}, 
cf.~Remark~\ref{r: no prenex} below), so we 
indicate their structure explicitly when necessary.

%\newpage

\section*{Finitary case}

In this section, we consider 
$\mathcal L_{\omega,\omega}$, the usual 
first-order finitary language. The theorem 
below characterizes sentences~$\varphi$ of 
$\mathcal L_{\omega,\omega}(\tau)$
that have $\vartheta(\varphi)$ expressible 
in the same $\mathcal L_{\omega,\omega}(\tau)$, 
assuming that the signature~$\tau$ has no 
function symbols. Note first that, for any 
$\varphi$ in such a~$\tau$, and any $n<\omega$, 
$\neg\,\vartheta_{\le n}(\varphi)$ is 
equivalent to a~$\Pi^{0}_1$-sentence of 
$\mathcal L_{\omega,\omega}(\tau)$, and therefore,  
$$
T_\varphi:=
\{\neg\,\vartheta_{\le n}(\varphi):n<\omega\}
$$
is a~$\Pi^{0}_1$-theory. Clearly, it expresses that
neither finite model of $\tau$ satisfies~$\varphi$; 
we just use the shorter notation $T_\varphi$ 
instead of $\neg\,\vartheta_{<\omega}(\varphi)$.

\begin{tm}\label{t: syntactic criterion}
Let $|\fnc(\tau)|=0$ and $\varphi$ a~sentence 
in $\mathcal L_{\omega,\omega}(\tau)$. 
The following are equivalent:
\begin{enumerate}
\setlength\itemsep{-0.2em}
\item[(i)] 
$\vartheta(\varphi)$~is expressible by a~sentence 
in $\mathcal L_{\omega,\omega}(\tau)$;
\item[(ii)] 
$T_\varphi\vDash\varphi\lra\psi$ 
for some $\Sigma^{0}_2$-sentence~$\psi$
in $\mathcal L_{\omega,\omega}(\tau)$.
\end{enumerate}
\end{tm}

\begin{proof}
(i)$\Rightarrow$(ii). 
Assume (ii)~fails, so $\neg\,\varphi$ is not 
equivalent to any $\Pi^{0}_2$-sentence under~$T_\varphi$. 
By the standard characterization of $\Pi^{0}_2$-formulas 
(see, e.g., \cite{Chang Keisler}), then there is 
an increasing chain of models of 
$T_\varphi\cup\{\neg\,\varphi\}$ such that 
their union $\mathfrak A$ satisfies~$\varphi$. 
Since $T_\varphi$ is $\Pi^{0}_1$, it is preserved 
under increasing chains, so $\mathfrak A$ satisfies 
$T_\varphi$ as well. Thus we have a~model $\mathfrak A$ 
such that it satisfies $\varphi$ but neither its finite 
submodel does. By \cite{Saveliev 2019}, this implies 
that $\varphi$ is not expressed by a~sentence in 
$\mathcal L_{\omega,\omega}(\tau)$.

(ii)$\Rightarrow$(i). 
Let (ii) hold, so there is 
a~$\Sigma^{0}_2$-sentence $\psi$ such that 
$T_\varphi$ proves $\varphi\lra\psi$. 
Let $\psi$ be of the form
$
\exists x_0\ldots\exists x_{k-1}
\forall y_0\ldots\forall y_{l-1}
\,\psi_0 
$
where $\psi_0$~is open (i.e., quantifier free). 
By the compactness theorem, there is 
a~finite fragment of~$T_\varphi$ which implies 
this equivalence, w.l.g.~we can assume that 
the fragment is just 
$\neg\,\vartheta_{\le n}(\varphi)$ 
for some $n<\omega$. Let us show now that 
every model $\mathfrak A\vDash\varphi$ 
includes a~submodel $\mathfrak B\vDash\varphi$ 
of cardinality~$\le\max(k,n)$.

Indeed, if 
$\mathfrak A\vDash\vartheta_{\le n}(\varphi)$, 
then, by the definition of~$\vartheta_{\le n}$,  
$\mathfrak A$ includes a~submodel 
$\mathfrak B\vDash\varphi$ of cardinality $|B|\le n$. 
And if 
$\mathfrak A\vDash\neg\,\vartheta_{\le n}(\varphi)$, 
then $\mathfrak A\vDash\varphi\lra\psi$, and so 
$\mathfrak A\vDash\psi$. Then, since $\psi$ 
is~$\Sigma^{0}_2$, $\mathfrak A$ includes a~submodel 
$\mathfrak B\vDash\psi$ of cardinality $|B|\le k$. 
Moreover, since $\neg\,\vartheta_{\le n}(\varphi)$ 
is~$\Pi^{0}_1$, we have 
$\mathfrak B\vDash\neg\,\vartheta_{\le n}(\varphi)$. 
Therefore, $\mathfrak B\vDash\varphi$. 
\end{proof}

\begin{cor}\label{c: sufficient}
Let $|\fnc(\tau)|=0$ and $\varphi$ a~sentence 
in $\mathcal L_{\omega,\omega}(\tau)$. 
Any of the following conditions on $\varphi$ 
is sufficient for $\vartheta(\varphi)$ to be 
expressible by a~sentence in 
$\mathcal L_{\omega,\omega}(\tau)$:
\begin{enumerate}
\setlength\itemsep{-0.2em}
\item[(i)] 
$\varphi\lra\psi$ for some $\psi$ in~$\Sigma^{0}_2$;
\item[(ii)] 
there is $n\in\omega\setminus1$ such that 
$\neg\,\varphi$ has no models of cardinality~$n$.
\end{enumerate}
Moreover,
\begin{enumerate}
\setlength\itemsep{-0.2em}
\item[(iii)] 
$\vartheta(\varphi)\lra\varphi$ iff 
$\varphi\lra\psi$ for some $\psi$ in~$\Sigma^{0}_1$;
\item[(iv)] 
$\vartheta(\varphi)\lra\top$ iff 
any finite model of~$\tau$ has 
a~submodel satisfying~$\varphi$ iff 
any model of~$\tau$ has 
a~finite submodel satisfying~$\varphi$.  
\end{enumerate}
\end{cor}

\begin{proof}
Clear. 
\end{proof}

%\newpage 

Let us now discuss expressibility by theories; 
we do not longer assume signatures to be purely  
predicate. As with expressibility by a~single 
sentence, we can say that $\vartheta(\varphi)$ 
is {\it expressible by a~theory~$T$} iff
for any model~$\mathfrak A$,
$$
\mathfrak A\vDash\vartheta(\varphi) 
\;\text{ iff }\;
\mathfrak A\vDash T. 
$$ 
Two propositions below say that, in modal terms, 
our ``possibility'' operator (i.e.,~$\vartheta$) 
applied to some sentences $\varphi$ of 
$\mathcal L_{\omega,\omega}$ can be expressed 
by a~theory but not by a~sentence of the same 
$\mathcal L_{\omega,\omega}$, while for the 
``necessity'' operator (i.e.,~$\neg\,\vartheta$) 
this situation never happens.

\begin{prp}\label{p: theta expressed by thy}
Let $\tau:=\{<\}$ where $<$~is a~binary predicate 
symbol. There is a~sentence~$\varphi$ of 
$\mathcal L_{\omega,\omega}(\tau)$ such that 
$\vartheta(\varphi)$ is expressible 
in $\mathcal L_{\omega,\omega}(\tau)$ 
by a~theory but not by any single sentence. 
\end{prp}

\begin{proof}
Let $\varphi$ say that, if $<$~is a~linear order, 
then it has either no first or no last element, 
and let $T:=\{\sigma_{n}:n<\omega\}$ where 
the sentence $\sigma_{n}$ says that, if $<$~is 
a~linear order, then it has~$\ge n$ elements. 
Then $\vartheta(\varphi)$ is expressible by the  
theory~$T$ (since, for any $\mathfrak A\vDash T$, 
if $\mathfrak A$ is a~linearly ordered set, 
then it is infinite, and so, it includes a~subset 
that has either no first element or no last element) 
but not by any single sentence (since $T$ is not 
equivalent to any of its finite fragments).
\end{proof}

\begin{rmk}\label{r: theta inexpressed by thy}
A~binary predicate symbol in $\tau$ of  
Proposition~\ref{p: theta expressed by thy} 
cannot be reduced to unary since for purely 
monadic~$\tau$, $\mathcal L_{\omega,\omega}(\tau)$ 
is closed under~$\vartheta$; an infinitary analog 
of this fact will be considered below. 
Note also that $\vartheta(\varphi)$ in 
Proposition~\ref{p: theta expressed by thy} 
is expressible by one sentence of 
$\mathcal L_{\omega_1,\omega}(\tau)$ 
(since expressible by a~countable theory). 
Example~\ref{e: inexpressible theta} below provides 
sentences $\varphi$ for which $\vartheta(\varphi)$ 
express some natural properties inexpressible in 
$\mathcal L_{\omega_1,\omega}(\tau)$, and even in 
$\mathcal L_{\infty,\omega}(\tau)$, and thus  
by any theory in the language of~$\varphi$. 
\end{rmk}

\hide{
By contrast, there is no sentence~$\varphi$ 
in $\mathcal L_{\omega,\omega}$ such that 
$\neg\,\vartheta(\varphi)$ would be expressible by 
an essentially infinite theory in the same language: 
}

\begin{prp}\label{p: neg-theta inexpressed by thy}
For any $\tau$ and $\varphi$ in 
$\mathcal L_{\omega,\omega}(\tau)$, 
$\neg\,\vartheta(\varphi)$ is expressible 
in $\mathcal L_{\omega,\omega}(\tau)$ by 
a~theory iff it is expressible by a~sentence. 
\end{prp}

\begin{proof} 
In order to prove the non-trivial implication, 
assume there is a~theory~$T$ in 
$\mathcal L_{\omega,\omega}(\tau)$ such that, 
for all models $\mathfrak A$ of~$\tau$,
$$
\mathfrak A\vDash\neg\,\vartheta(\varphi)
\;\text{ iff }\;
\mathfrak A\vDash T 
$$ 
(i.e., $\mathfrak A\vDash T$ means that
every submodel~$\mathfrak B$ of $\mathfrak A$ 
satisfies~$\neg\,\varphi$), and show that such 
a~$T$ is axiomatized by a~single sentence~$\psi$.

Let $\tau':=\tau\cup\{P\}$ where $P$~is a~new 
unary predicate symbol, and let $T':=T\cup T_P$ 
where $T_P$ is a~theory stating that $P$~forms 
a~submodel, i.e., that it is closed under all 
function (including constant) symbols of~$\tau$. 
(Of course, if $|\fnc(\tau)|\ge\omega$, then 
$T_P$~is rather theory than a~single sentence, 
and $T'$~is consistent whenever so is~$T$.)  

The theory $T'\cup\{\varphi^P\}$ is inconsistent. 
(Indeed, otherwise pick any its model~$\mathfrak A'$; 
if $\mathfrak A$ is the $\tau$-reduct of~$\mathfrak A'$, 
then $\mathfrak A$ satisfies $T$ and includes 
a~submodel $\mathfrak B$ satisfying~$\varphi$, 
namely, its universe~$B$ is~$P^{\mathfrak A'}$; 
this contradicts to the assumption that $T$ expresses 
the non-existence of submodels satisfying~$\varphi$.)  
So by compactness, there is a~finite fragment 
$T'_0$ of~$T'$ that implies $\neg\,\varphi^P$. Let 
$T_0:=T'_0\setminus T_P$ and $\psi:=\bigwedge T_0$. 
Then $\psi$ axiomatizes~$T$. (Indeed, pick any model 
$\mathfrak A\vDash\psi$ and show that any its 
submodel $\mathfrak B$ satisfies~$\neg\,\varphi$. 
If $\mathfrak A'$ is the $\tau'$-expansion of 
$\mathfrak A$ with $P^{\mathfrak A'}:=B$, then 
$\mathfrak A'\vDash T'_0$ and hence 
$\mathfrak A'\vDash\neg\,\varphi^P$, thus 
$\mathfrak B\vDash\neg\,\varphi$, as required.) 
\end{proof}

%\newpage

\section*{Infinitary case}

In this section, we consider languages 
$\mathcal L_{\kappa,\lambda}$ with 
$\kappa>\omega$. First, given a~signature 
$\tau$ and cardinals $\kappa,\lambda,\nu$,  
% with $\omega\le\lambda\le\kappa$, 
we calculate $\kappa',\lambda'$ 
such that whenever $\varphi$ is in 
$\mathcal L_{\kappa,\lambda}(\tau)$ then 
$\vartheta_{\le\nu}(\varphi)$ is expressible 
by an (explicitly written) sentence in 
$\mathcal L_{\kappa',\lambda'}(\tau)$; 
thence we obtain sufficient conditions 
for $\mathcal L_{\kappa,\lambda}(\tau)$
to be closed under~$\vartheta_{\le\nu}$. 
Next, we present some infinitary versions 
of the downward L{\"o}wenheim--Skolem 
theorem and use them to isolate cases when 
$\vartheta$ and $\vartheta_{\le\nu}$ coincide. 
Finally, combining both observations, we establish 
when $\vartheta(\varphi)$ is expressible in an 
appropriate $\mathcal L_{\kappa',\lambda'}(\tau)$; 
thence we obtain conditions sufficient for 
$\mathcal L_{\kappa,\lambda}(\tau)$ to be closed 
under~$\vartheta$, noticing incidentally that for 
any given $\lambda$ and $\tau$ there are properly 
many such $\mathcal L_{\kappa,\lambda}(\tau)$.

Recall that $\vartheta_{\le\nu}(\varphi)$ 
expresses the existence of a~submodel having  
cardinality~$\le\nu$ and satisfying~$\varphi$; 
note that $\vartheta_{\le\nu}(\varphi)$ can be 
understood as 
$\vartheta(\varphi\wedge\sigma_{\le\nu})$ 
where the sentence $\sigma_{\le\nu}$ asserts that 
there are~$\le\nu$ elements (it belongs to 
$\mathcal L_{\nu^+,\nu^+}(\{=\})$, see 
\cite{Dickmann 1975}, Example~1.2.18); 
likewise for~$\vartheta_{\nu}(\varphi)$.  
 %\hide{\DS{
 %[NB: In \cite{Saveliev 2019}], 
 %the notation provided 
 %a~($\le\nu$)-generated model; here the latter is 
 %denoted by $\vartheta_{\le\nu\gen}(\varphi)$.]}
 %}

Given a~signature~$\tau$, a~formula~$\varphi$ 
of $\mathcal L_{\kappa,\lambda}(\tau)$, and 
a~set $\{x_\alpha\}_{\alpha<\nu}$ of first-order 
variables, we define the {\em relativization} 
$\varphi^{\{x_\alpha\}_{\alpha<\nu}}$ of $\varphi$ 
to $\{x_\alpha\}_{\alpha<\nu}$ by recursion:
\begin{enumerate}
\setlength\itemsep{-0.2em}
\item[(i)] 
  \hide{
  if $\varphi$ is atomic, say $P(t_0,\ldots,t_n)$ 
  where $t_0,\ldots,t_n$ are terms whose all 
  free variables are $y_0,\ldots,y_{k-1}$, 
  then $\varphi^{\{x_\alpha\}_{\alpha<\nu}}$ is 
  $
  \varphi\wedge
  \bigwedge_{i<k}
  \bigvee_{\!\alpha<\nu}
  y_i=x_\alpha
  $;
  %%% it seems, this is redundant; 
  %%% the reasonable definition is as follows;
}
if $\varphi$ is atomic, 
then $\varphi^{\{x_\alpha\}_{\alpha<\nu}}$ 
is $\varphi$;
\item[(ii)] 
if $\varphi$ is $\neg\,\psi$, then 
$\varphi^{\{x_\alpha\}_{\alpha<\nu}}$ is 
$\neg\,\psi^{\{x_\alpha\}_{\alpha<\nu}}$; 
\item[(iii)] 
if $\varphi$ is $\bigwedge_{\beta<\gamma}\psi_\beta$, 
then $\varphi^{\{x_\alpha\}_{\alpha<\nu}}$ is 
$
\bigwedge_{\beta<\gamma}
\psi^{\{x_\alpha\}_{\alpha<\nu}}_\beta
$; 
\item[(iv)] 
if $\varphi$ is 
$\bigexists_{\beta<\delta}y_\beta\,\psi$, 
then $\varphi^{\{x_\alpha\}_{\alpha<\nu}}$ is 
$
\bigexists_{\beta<\delta}y_\beta\, 
\bigl(
\psi^{\{x_\alpha\}_{\alpha<\nu}}
\wedge 
\bigwedge_{\beta<\delta}
\bigvee_{\!\alpha<\nu}
y_\beta=x_\alpha
\bigr)
$. 
\end{enumerate}
Note that all relativizations of an open 
formula~$\varphi$ coincide with $\varphi$ itself.

\begin{lm}\label{l: relativization}
Let $\varphi$ be a~sentence in 
$\mathcal L_{\kappa,\lambda}(\tau)$, 
and let $\mu:=|\fnc(\tau)|$. Then 
$\vartheta_{\le\nu}(\varphi)$ 
(and $\vartheta_\nu(\varphi)$)
is expressible by a~sentence in 
$\mathcal L_{\kappa',\lambda'}(\tau)$ where 
$\kappa':=\max(\kappa,\mu^+,\nu^+)$ and 
$\lambda':=\max(\lambda,\nu^+)$.
\end{lm}

\begin{proof}
Let $\psi(x_\alpha)_{\alpha<\nu}$ express that 
$\{x_\alpha\}_{\alpha<\nu}$ forms a~submodel, 
i.e., it is the formula 
$$
\bigwedge_{\beta<\mu}\;
\bigwedge_{f\in\nu^{n_\beta}}\;
\bigvee_{\gamma<\nu}\,
F_\beta(x_{f(0)},\ldots,x_{f(n_\beta-1)})
=x_\gamma
$$
where $F_\beta$, $\beta<\mu$, enumerate 
$\fnc(\tau)$, and $n_\beta$~is the arity 
of~$F_\beta$. Then $\vartheta_{\le\nu}(\varphi)$ 
is equivalent to the sentence 
$$
\bigexists_{\alpha<\nu}x_\alpha\,
\bigl(
\psi(x_\alpha)_{\alpha<\nu}
\wedge
\varphi^{\{x_\alpha\}_{\alpha<\nu}}
\bigr).
$$
Clearly, 
$\psi(x_\alpha)_{\alpha<\nu}$ is in 
$\mathcal L_{\kappa',\omega}(\fnc(\tau))$, 
and $\varphi^{\{x_\alpha\}_{\alpha<\nu}}$ is in 
$\mathcal L_{\max(\kappa,\nu^+),\lambda'}(\tau)$. 

For $\vartheta_\nu(\varphi)$, 
it suffices to add the formula  
$
\neg\!\bigvee_{\alpha<\beta<\nu}
x_\alpha=x_\beta
$
to the conjuction above. 
\end{proof}

\begin{cor}\label{c: expressibility of theta_nu}
Let $\omega\le\lambda\le\kappa$, 
$\mu:=|\fnc(\tau)|<\kappa$, $\nu<\lambda$.  
Then $\mathcal L_{\kappa,\lambda}(\tau)$ is closed 
under $\vartheta_{\le\nu}$ (and $\vartheta_\nu$). 
\end{cor}

\begin{proof}
By Lemma~\ref{l: relativization} 
since $\mu<\kappa$, $\nu<\lambda\le\kappa$ 
give $\kappa'=\kappa$, $\lambda'=\lambda$.
\end{proof}

%\newpage

The following theorem summarizes some infinitary 
downward L{\"o}wenheim--Skolem-type results. 
Given a~set $\varGamma$ of formulas in 
$\mathcal L_{\kappa,\kappa}(\tau)$, let 
$\Sbf(\varGamma)$ be the set of subformulas 
of all formulas in~$\varGamma$, and let 
$\prec_\varGamma$ denote $\varGamma$-elementary 
embedding. For brevity, we write just 
$\Sbf(\varphi)$ instead of $\Sbf(\{\varphi\})$.

\begin{tm}\label{t: dLS vs theta}
Let $\tau$ be a~signature, $\mathfrak A$ 
a~model of~$\tau$, $\mu:=|\fnc(\tau)|$, 
$\kappa$~a~regular cardinal,
and $\omega\le\lambda\le\kappa$. 
\begin{enumerate}
\setlength\itemsep{-0.2em}
\item[(i)] 
Assume $X\subseteq A$,  
$\max(2,|X|,\mu)\le\nu=\nu^{<\kappa}\le|A|$, 
and $\varGamma$ is a~set of formulas in 
$\mathcal L_{\kappa,\kappa}(\tau)$ such that 
$|\Sbf(\varGamma)|\le\nu$. Then $\mathfrak A$ 
has a~submodel~$\mathfrak B$ such that 
$X\subseteq B$, $|B|\le\nu$, and 
$\mathfrak B\preceq_{\varGamma}\mathfrak A$.
\item[(ii)] 
Assume $\lambda$~is regular, $\mu<\kappa$, 
and $\varphi$~is a~sentence in 
$\mathcal L_{\kappa,\lambda}(\tau)$ 
such that $\mathfrak A\vDash\varphi$. 
Then $\mathfrak A$ has a~submodel 
$\mathfrak B$ such that $|B|<\kappa$ and 
$\mathfrak B\vDash\varphi$, moreover, 
$
\mathfrak B
\preceq_{\Sbf(\varphi)}\mathfrak A,  
$ 
whenever one of (a)--(c) holds: 
\begin{enumerate}
\setlength\itemsep{-0.2em}
\item[(a)] 
$\lambda<\kappa$ and  
${\kappa_0}^{\lambda_0}<\kappa$, for all 
$\lambda_{0}<\lambda$ and $\kappa_0<\kappa$; 
\item[(b)] 
$\kappa>\omega$~is strongly inaccessible 
and $\lambda$~is regular;
\item[(c)] 
$\kappa$~is not the successor of a~singular 
cardinal and $\lambda$~is strongly compact.
\end{enumerate}
\end{enumerate}
\end{tm}

\begin{proof}
(i). 
See \cite{Dickmann 1975}, Corollary~3.4.2, 
in which we weakened the condition 
$\mu:=|\tau|$ to $\mu:=|\fnc(\tau)|$ 
(this can be done since it can also be done 
with Theorem~3.4.1, as a~careful analysis of 
its proof shows) and relativize the conclusion 
to~$\varGamma$ (we can have the full 
$\mathcal L_{\kappa,\kappa}(\tau)$-elementarity  
under the extra assumption $|\tau|\le\nu$). 

\begin{rmk}
Alternatively, we can argue in terms of reducts. 
Let $\tau_\varGamma:=\fnc(\tau)\cup\tau\uhr\varGamma$ 
where $\tau\uhr\varGamma$ is the set of symbols 
occuring in~$\varGamma$. Apply the appropriate 
L{\"o}wenheim--Skolem theorem to 
the $\tau_\varGamma$-reduct of $\mathfrak A$. 
Since $\tau_\varGamma$ contains all function symbols 
of~$\tau$, the universe of the resulting submodel 
is closed under all functions of~$\tau$ and so 
determines a~$\varGamma$-elementary $\tau$-submodel 
of~$\mathfrak A$. 
\end{rmk}

(ii). 
For~(a), see \cite{Dickmann 1975}, Corollary~3.4.4, 
with the same minor improvement in the value of~$\mu$ 
(additionally, we can have 
$
\mathfrak B
\preceq_{\mathcal L_{\omega,\omega}(\tau)}\!
\mathfrak A 
$
whenever $|\tau|<\kappa$.)

For~(b), if $\lambda<\kappa$, this follows 
from~(a); and if $\lambda=\kappa$, it suffices 
to note that $\mathcal L_{\kappa,\kappa}$ is 
the union of $\mathcal L_{\kappa,\nu^+}$, 
$\nu<\kappa$, each of which satisfies~(a). 

For~(c), if $\kappa=\lambda$, we clearly have
the conditions of~(b). And if $\kappa>\lambda$, 
we have the conditions of~(a), i.e., 
$
(\forall\lambda_{0}<\lambda)
(\forall\kappa_0<\kappa)\,
{\kappa_0}^{\lambda_0}<\kappa. 
$ 
Indeed, recall that, since $\lambda$~is 
strongly compact, $\rho^{<\lambda}=\rho$ 
for all regular $\rho\ge\lambda$ (see, 
e.g., \cite{Jech}, Lemma~20.11). Now, if 
$\kappa=\rho^+$ for some regular~$\rho$, then 
$
\kappa_{0}^{\lambda_0}\le
\rho^{<\lambda}=\rho<\kappa,
$
and if $\kappa$~is limit, 
$
\kappa_{0}^{\lambda_0}\le
(\kappa_{0}^+)^{<\lambda}<\kappa.
$
\end{proof}

%\newpage

\begin{rmk}\label{r: properly many kappa}
The conditions of (ii)(a), which serves as 
a~basis for (ii)(b) and (ii)(c), hold for 
a~plenty of cardinals (in $\ZFC$ alone). 
Namely, for any $\lambda\ge\omega$ there are 
properly many $\kappa$ satisfying these conditions, 
e.g., so are $\kappa:=(\nu^{<\lambda})^+$ 
for all $\nu\ge\max(2,\mu)$. 
In particular, these conditions hold 
with $\lambda=\omega$ for all $\kappa>\omega$.
\end{rmk}

%\newpage

In Corollary~\ref{c: dLS vs theta}, we provide 
conditions sufficient for coincidence of 
$\vartheta$ and its bounded 
version~$\vartheta_{\le\nu}$ that 
clearly follow from Theorem~\ref{t: dLS vs theta} 
(in the particular case of $\kappa=\lambda$, 
(iic)~was stated in \cite{Saveliev 2019}, 
Theorem~10; moreover, it was pointed out 
in \cite{Saveliev 2019}, Remark~9 that  
in this case (iib)~holds).

\begin{cor}\label{c: dLS vs theta}
Let $\tau$ be a~signature, $\mu:=|\fnc(\tau)|$, 
$\kappa$~a~regular cardinal, and 
$\omega\le\lambda\le\kappa$.  
\begin{enumerate}
\setlength\itemsep{-0.2em}
\item[(i)] 
Assume $\max(2,\mu)\le\nu=\nu^{<\kappa}$. 
Then for any sentence~$\varphi$ 
in $\mathcal L_{\kappa,\kappa}(\tau)$,
$$
\vartheta(\varphi)
\;\text{ iff }\;
\vartheta_{\le\nu}(\varphi). 
$$
%on all models of~$\tau$.
\item[(ii)] 
Assume $\lambda$~is regular and $\mu<\kappa$. 
Then for any sentence~$\varphi$ 
in $\mathcal L_{\kappa,\lambda}(\tau)$,  
$$
\vartheta(\varphi)
\;\text{ iff }\;
\vartheta_{<\kappa}(\varphi) 
$$
whenever one of (a)--(c) of 
Theorem~\ref{t: dLS vs theta}(ii) holds.
\hide{
\begin{enumerate}
\item[(a)] 
$\lambda<\kappa$ and  
${\kappa_0}^{\lambda_0}<\kappa$, for all 
$\lambda_{0}<\lambda$ and $\kappa_0<\kappa$;
\item[(b)] 
$\kappa>\omega$~is strongly inaccessible 
and $\lambda$~is regular;
\item[(c)] 
$\kappa$~is not the successor of a~singular 
cardinal and $\lambda$~is strongly compact.
\end{enumerate}
}
\end{enumerate}
\end{cor}

\begin{proof}
Immediate from Theorem~\ref{t: dLS vs theta}.
\end{proof}

\begin{rmk}
It follows from Theorem~\ref{t: dLS vs theta} 
that Corollary~\ref{c: dLS vs theta} remains 
true if we replace $\vartheta$ associated with 
submodels by its analog associated with 
$\mathcal L_{\kappa,\kappa}(\tau)$-elementary 
submodels in~(i) 
under the extra assumption $|\tau|\le\nu$, and 
with $\mathcal L_{\omega,\omega}(\tau)$-elementary 
submodels in~(ii) 
under the extra assumption $|\tau|<\kappa$. 
\end{rmk}

%\newpage 

Now we are able to find out, for a~given 
language $\mathcal L_{\kappa,\lambda}(\tau)$ 
and an its sentence~$\varphi$, in which language 
$\mathcal L_{\kappa',\lambda'}(\tau)$ the 
sentence $\vartheta(\varphi)$ falls into, 
and to formulate thence conditions sufficient 
for $\mathcal L_{\kappa,\lambda}(\tau)$ 
to be closed under~$\vartheta$. 
For this, we combine conditions for  
the equivalence of $\vartheta(\varphi)$ and 
$\vartheta_{\le\nu}(\varphi)$, provided by 
Corollary~\ref{c: dLS vs theta}, and for the 
expressibility of $\vartheta_{\le\nu}(\varphi)$, 
provided by Lemma~\ref{l: relativization} and 
by Corollary~\ref{c: expressibility of theta_nu}. 

\hide{
\begin{rmk}\label{r: incompatibility of conditions}
Let us point out that the conditions of 
Corollary~\ref{c: expressibility of theta_nu}, 
which guarantees that 
$\mathcal L_{\kappa,\lambda}(\tau)$ is closed 
under~$\vartheta_{\le\nu}$, include $\nu<\lambda$ 
and, therefore, are incompatible with the conditions 
of Corollary~\ref{c: dLS vs theta}(i) (which include 
$\nu=\nu^{<\kappa}$ and $\kappa=\lambda$, thus 
$\nu\ge\kappa$) as well as with the special case of 
Corollary~\ref{c: dLS vs theta}(ii) in which we have 
$\kappa>\lambda=\omega$ (since $\vartheta_{<\kappa}$ 
should be $\vartheta_{\le\nu}$ for some 
$\nu\ge\omega$). However, other cases of 
Corollary~\ref{c: dLS vs theta}(ii) fit well. 
\DS{[remove this remark from the finalized text?]}
\DS{[2026: I did nto check wheteher this remark 
remans correct after numerous corrections made 
in the renewed version...]}
\end{rmk}
}

\begin{tm}\label{t: expressing theta}
Let $\tau$~be a~signature, 
$\mu:=|\fnc(\tau)|$, 
$\kappa$~a~regular cardinal, and 
$\omega\le\lambda\le\kappa$.  
\begin{enumerate}
\setlength\itemsep{-0.2em}
\item[(i)] 
Assume $\max(2,\mu)\le\nu=\nu^{<\kappa}$ and 
$\varphi$ is in $\mathcal L_{\kappa,\kappa}(\tau)$. 
Then $\vartheta(\varphi)$ is 
in $\mathcal L_{\nu^+,\nu^+}(\tau)$.
\item[(ii)] 
Assume $\mu<\kappa$ and $\kappa$ 
is strongly inaccessible. 
Then $\mathcal L_{\kappa,\kappa}(\tau)$ 
is closed under~$\vartheta$.  
\end{enumerate}
\end{tm}

\begin{proof} 
(i). 
By Corollary~\ref{c: dLS vs theta}(i), 
$\vartheta(\varphi)$ is equivalent to 
$\vartheta_{\le\nu}(\varphi)$. 
And by Lemma~\ref{l: relativization},
$\vartheta_{\le\nu}(\varphi)$ 
is equivalent to a~formula in 
$\mathcal L_{\kappa',\lambda'}(\tau)$ 
where $\kappa':=\max(\kappa,\mu^+,\nu^+)$, 
$\lambda':=\max(\lambda,\nu^+)$.  
Since $\kappa\le\nu,\mu\le\nu$, we get  
$\kappa'=\lambda'=\nu^+$, as required. 

(ii). 
Assume first that $\varphi$ is in 
$\mathcal L_{\kappa,\lambda}(\tau)$ 
for some regular $\lambda\le\kappa$ and one of 
(a)--(c) of Theorem~\ref{t: dLS vs theta}(ii) holds. 
Then by Corollary~\ref{c: dLS vs theta}(ii), 
$\vartheta(\varphi)$ is equivalent to 
$\vartheta_{<\kappa}(\varphi)$, so  
there is $\nu<\kappa$ such that $\vartheta(\varphi)$ 
is equivalent to $\vartheta_{\nu}(\varphi)$. 
Now, if $\kappa$ is strongly inaccessible 
(i.e., (b)~holds) and $\lambda=\kappa$, we can 
apply Corollary~\ref{c: expressibility of theta_nu} 
to conclude that for each $\nu<\kappa$, 
$\vartheta_{\nu}(\varphi)$ is equivalent to 
a~formula in the same language 
$\mathcal L_{\kappa,\lambda}(\tau)$, as required.
\end{proof}

\hide{
\begin{rmk} 
If $\kappa$ is limit (and so weakly inaccessible), 
$\vartheta(\varphi)$ is not expressed in 
$\mathcal L_{\kappa,\lambda}(\tau)$ {\it uniformly} 
in~$\varphi$ since $\nu<\kappa$ in the equivalent 
formula $\vartheta_{\nu}(\varphi)$ depend 
on~$\varphi$ (while $\vartheta_{<\kappa}(\varphi)$ is 
$\bigvee_{\nu<\kappa}\vartheta_{\nu}(\varphi)$, which 
is too large to be in $\mathcal L_{\kappa,\lambda}$). 
On the other hand, for $\kappa^+$, $\vartheta(\varphi)$ 
is expressed by $\vartheta_{\le\kappa}(\varphi)$ 
uniformly for all~$\varphi$.  
\DS{[Give examples of non-uniformity!]}
\DS{[2026: check]}
\end{rmk}
}

%\newpage

\begin{ex}\label{e: inexpressible theta}
We provide a~few sentences~$\varphi$ for which 
$\vartheta(\varphi)$ express certain natural 
properties but are inexpressible in the same 
language as $\varphi$~itslef. Let $\tau$ consist 
of a~binary predicate symbol; then, assuming for 
simplicity that the order theory holds, we have:  
\begin{enumerate}
\setlength\itemsep{-0.2em}
\item[(i)] 
if $\varphi$ is 
a~$\mathcal L_{\omega,\omega}(\tau)$-sentence 
that says that there is no minimal element, 
then $\neg\,\vartheta(\varphi)$ expresses 
the {\it well-foundedness}  
(this is essentially the example of  
\cite{Saveliev 2019}, Theorem~4);
\item[(ii)] 
If $\psi$ is 
a~$\mathcal L_{\omega,\omega}(\tau)$-sentence 
that expresses that the order is dense and 
non-trivial (i.e., has more than one element), 
then $\neg\,\vartheta(\psi)$ expresses 
its {\it scatteredness}; 

\item[(iii)] 
if $\chi$ is 
a~$\mathcal L_{\omega_1,\omega}(\tau)$-sentence 
(inexpressible in $\mathcal L_{\omega,\omega}(\tau)$ 
of course) that says that the ordered set is 
infinite and that it is an antichain 
(i.e., the order is empty), then 
$\neg\,\vartheta(\varphi\vee\chi)$, 
where $\varphi$~is as in~(i), expresses  
the {\it well-quasi-orderedness}.
\end{enumerate}
None of the three properties is expressible in 
$\mathcal L_{\omega_1,\omega}(\tau)$; in fact, 
the properties in (i) and~(iii) are inexpressible 
even in $\mathcal L_{\infty,\omega}(\tau)$ 
(see \cite{Saveliev 2019}, Theorem~4 
for a~slightly more general observation). 
%% \DS{[What about~(ii)?]}
On the other hand, each of them is expressible 
in $\mathcal L_{\omega_1,\omega_1}(\tau)$. 
For (i) and~(ii), this follows from 
Theorem~\ref{t: expressing theta}(i); 
for~(iii), by a~direct argument: 
any suborder with no minimal element contains 
a~countable suborder with no minimal element, 
and any infinite antichain contains a~countable 
antichain, therefore,  
$
\vartheta(\varphi\vee\chi)\lra
\vartheta_{\le\omega}(\varphi\vee\chi)
$, 
and Lemma~\ref{l: relativization} 
gives us the expressibility in 
$\mathcal L_{\omega_1,\omega_1}(\tau)$. 
\end{ex}

%\newpage

\section*{Monadic languages}

In this section, first we consider arbitrary 
languages $\mathcal L_{\kappa,\lambda}$ and 
signatures~$\tau$ but formulas of a~rather simple  
structure that generalize $\Sigma^{0}_2$-formulas 
of $\mathcal L_{\omega,\omega}(\tau)$, and show 
that sentences of this form have the so-called 
L{\"o}wenheim number~$<\kappa$ (i.e., whenever 
such a~sentence is satisfied in a~model then 
it is satisfied in an its submodel of 
cardinality~$<\kappa$). We conclude that 
$\vartheta$ and $\vartheta_{<\kappa\gen}$ 
(or $\vartheta_{<\kappa}$ for $\tau$~without 
function symbols) coincide on these sentences.  
Further, we show then that so-called monadic-like 
sentences have this form; moreover, any such 
sentence of $\mathcal L_{\kappa,\lambda}(\tau)$ 
is equivalent to a~Boolean combination of 
existential (or universal) one-variable 
sentences of $\mathcal L_{\kappa,\omega}(\tau)$ 
under certain conditions on $\kappa,\lambda$, 
and $\tau$ (namely, either $\kappa=\lambda$ is 
strongly inaccessible or $2^{2^{|\tau|}}<\kappa$). 
Applying these observations to purely monadic 
(i.e., consisting only of unary predicate 
symbols) signatures~$\tau$ without equality, 
we prove that in this case 
$\mathcal L_{\kappa,\lambda}(\tau)$ 
is semantically equivalent to 
$\mathcal L_{\kappa,\omega}(\tau)$ 
and is closed under~$\vartheta$.  

We start from the following observation. 

\begin{lm}\label{l: E and A}
In any $\mathcal L_{\kappa,\lambda}(\tau)$,
universal sentences are preserved under 
submodels, and existential under extensions. 
The converse is also true whenever  
$\kappa=\omega$, or $\kappa=\lambda$ 
is strongly compact, or else 
$\kappa=\omega_1$ and $\lambda=\omega$.
\end{lm}

\begin{proof}
If $\kappa=\omega$, this is a~basic fact of 
model theory (see, e.g.,~\cite{Chang Keisler}, 
Proposition~3.2.2); if $\kappa=\lambda$ 
is strongly compact, modify the argument 
for~$\omega$ (this was already pointed out 
in \cite{Saveliev 2019}, Corollary~3); 
for the last case $\kappa=\omega_1$ and 
$\lambda=\omega$, see \cite{Dickmann 1975}, 
Corollary~2.3.9 (or Theorem~2.3.8). 
\end{proof}

\begin{rmk}\label{r: no prenex}
Let $\tau$ consist of the only equality symbol. 
In $\mathcal L_{\kappa^+,\kappa^+}(\tau)$,
there are sentences that are preserved under 
submodels but not equivalent to universal ones 
(see \cite{Dickmann 1975}, Counterexample~2.4.11). 
Furthermore, if $\kappa$~is limit, then 
$\mathcal L_{\kappa^+,\kappa}(\tau)$ 
contains a~sentence preserved under submodels 
but not equivalent to any sentence that is 
universal, and even a~conjunction of prenex 
$\mathcal L_{\infty,\kappa}(\tau)$-sentences
(see \cite{Dickmann 1975}, Counterexample~2.4.12). 
%% \DS{[Remove or rewrite??]}
\hide{
If $0<\alpha<\aleph_\alpha=\kappa$ for 
some~$\alpha$, then this sentence is even 
in $\mathcal L_{\kappa,\kappa}(\emptyset)$; 
but I hide this info since here I assume 
thorough that $\kappa$ is regular.
}
\end{rmk}

It is well known that any model of 
a~$\Sigma^{0}_2$-sentence~$\varphi$ of 
$\mathcal L_{\omega,\omega}(\tau)$ (with 
an arbitrary~$\tau$) contains a~finite submodel 
whose size depends on the number of existential 
quantifiers in~$\varphi$ 
(see, e.g., \cite{Borger et al}, Theorem~6.2.2). 
This statement, which can be considered as 
a~specific version of the downward 
L\"owenhein--Skolem theorem for such sentences, 
can be generalized to arbitrary infinitary 
languages $\mathcal L_{\kappa,\lambda}(\tau)$. 
According into account that the prenex form 
theorem generally fails in infinitary languages 
(see \cite{Dickmann 1975}, Counterexample~2.4.12 
and Appendix~B), we prove this generalization for 
a~slightly larger class consisting of disjunctions 
of conjunctions of (prenex) existential-universal 
sentences, i.e., of sentences of the form 
$
\bigexists_{\alpha<\gamma}x_\alpha
\bigforall_{\!\beta<\delta}y_\beta\;
\psi
$
where $\psi$~is open.

%\newpage

\begin{prp}\label{p: EA}
Let $\varphi$ be a~sentence 
in $\mathcal L_{\kappa,\lambda}(\tau)$
of the form 
$$
\bigvee_{\varepsilon<\zeta}\;
\bigwedge_{\eta<\xi_{\varepsilon}}\;
\varphi_{\varepsilon,\eta}
$$
where each sentence $\varphi_{\varepsilon,\eta}$ 
is existential-universal, say,    
$$
\bigexists_{\alpha<\gamma_{\varepsilon,\eta}}\!\!
x_{\varepsilon,\eta,\alpha}\!
\bigforall_{\beta<\delta_{\varepsilon,\eta}}\!\!
y_{\varepsilon,\eta,\beta}\;\,
\psi_{\varepsilon,\eta}
$$
with an open~$\psi_{\varepsilon,\eta}$\,, 
and for all $\varepsilon<\zeta$ let 
$
\nu_{\varepsilon}:=
\max\bigl(1+\sum_{\eta<\xi_{\varepsilon}}
|\gamma_{\varepsilon,\eta}|\bigr). 
$
Then any model $\mathfrak A$ satisfying~$\varphi$ 
includes a~$({\le}\nu_{\varepsilon})$-generated 
submodel~$\mathfrak B$ also satisfying~$\varphi$, 
for some $\varepsilon<\zeta$. 
If, moreover, $\fnc(\tau)$ is empty, 
then there is such a~$\mathfrak B$ of 
cardinality~$\le\nu_{\varepsilon}$. 

More generally, if $\varphi$ in 
$\mathcal L_{\kappa,\lambda}(\tau)$ is 
a~combination of conjunctions and disjunctions 
of existential-universal sentences, then 
the conclusion remains true with some 
$\nu<\kappa$ (the same for all models).
\end{prp}

\begin{proof}
Let $\mathfrak A$ be a~model of~$\tau$ such that 
$\mathfrak A\vDash\varphi$, so there is 
$\varepsilon<\zeta$ such that 
$
\mathfrak A\vDash
\bigwedge_{\eta<\xi_{\varepsilon}}
\varphi_{\varepsilon,\eta}.
$
Fix such an~$\varepsilon$. To simplify notation, 
we omit its symbol in the indices, thus letting 
$\xi:=\xi_{\varepsilon}$\,,
$\varphi_{\eta}:=\varphi_{\varepsilon,\eta}$\,, 
$x_{\eta,\alpha}:=x_{\varepsilon,\eta,\alpha}$\,, 
$y_{\eta,\beta}:=y_{\varepsilon,\eta,\beta}$\,, 
and $\nu:=\nu_\varepsilon$\,. 
In this notation, we have 
$
\mathfrak A\vDash
\bigwedge_{\eta<\xi}\varphi_{\eta}\,,
$
i.e., 
$$
\mathfrak A\,\vDash\,
\bigwedge_{\eta<\xi}\;
\bigexists_{\alpha<\gamma_{\eta}}\!\!
x_{\eta,\alpha}\!\!
\bigforall_{\beta<\delta_{\eta}}\!\!
y_{\eta,\beta}\;\,
\psi_{\eta}\,.
$$
We can assume w.l.g.~that $\gamma_\eta>0$; 
otherwise, a~dummy variable should be added.
For each $\eta<\xi$, pick elements 
$b_{\eta,\alpha}$, $\alpha<\gamma_{\eta}$, 
of~$A$ (the universe of the model~$\mathfrak A$) 
such that 
$$
\mathfrak A\,\vDash
\bigforall_{\beta<\delta_{\eta}}\!\!
y_{\eta,\beta}\;\,
\psi_{\eta}
\qquad
[b_{\eta,\alpha}/x_{\eta,\alpha}].
$$
Let $\mathfrak B$ be the submodel of $\mathfrak A$ 
generated by 
$
\{b_{\eta,\alpha}:
\eta<\xi,\alpha<\gamma_{\eta}\}.
$ 
Clearly, $\mathfrak B$ is $(\le\nu)$-generated;  
moreover, it has 
cardinality~$\le\nu$ if $\fnc(\tau)$ is empty. 
(Note that $\nu<\kappa$ since 
$\nu=\sum_{\eta<\xi}|\gamma_{\eta}|$ where  
$\xi<\kappa$ and $\gamma_{\eta}<\lambda\le\kappa$ 
for all $\eta<\xi$, and $\kappa$~is regular).

Treating $b_{\eta,\alpha}$ as new constant symbols, 
and thus
$
\bigforall_{\beta<\delta_{\eta}}
y_{\eta,\beta}\;
\psi_{\eta}\;
[b_{\eta,\alpha}/x_{\eta,\alpha}]
$
as a~universal sentence, we see that 
it is preserved in submodels (by 
the easy part of Lemma~\ref{l: E and A}), 
so $\mathfrak B$~satisfies it. Hence, 
$$
\mathfrak B\,\vDash\,
\bigexists_{\alpha<\gamma_{\eta}}\!\!
x_{\eta,\alpha}\!\!
\bigforall_{\beta<\delta_{\eta}}\!\!
y_{\eta,\beta}\;\,
\psi_{\eta}\,, 
$$
i.e., $\mathfrak B\vDash\varphi_{\eta}$\,,
and a~fortiori, $\mathfrak B\vDash\varphi$, 
as required. 

The proof of the general case is essentially 
the same. If $\mathfrak A\vDash\varphi$ where 
$\varphi$ is constructed of existential-universal 
sentences by conjunctions and disjunctions, 
we select one index in each of these disjunctions 
such that, by eliminating all disjunctions and  
leaving only conjunctions with the selected 
indices, we obtain a~sentence (obviously stronger 
than~$\varphi$ but) still satisfied in~$\mathfrak A$. 
We can then rewrite the remaining conjunctions 
as a~single one (which will be still of 
size~$<\kappa$), thus arriving at the case 
discussed above. 
Letting
$\nu:=\sup_{\varepsilon<\zeta}\nu_\varepsilon$\,, 
we see that $1\le\nu<\kappa$ and $\nu$~is 
the same for all models since it depends only 
for the structure of~$\varphi$.
\end{proof}

%\newpage 

\begin{cor}\label{c: EA}
If $\varphi$ in 
$\mathcal L_{\kappa,\lambda}(\tau)$ is 
a~combination of conjunctions and disjunctions 
of existential-universal sentences, then 
$\vartheta(\varphi)$ is equivalent to 
$\vartheta_{<\kappa\gen}(\varphi)$. 
If, moreover, $|\fnc(\tau)|=0$, 
then $\vartheta(\varphi)$ is equivalent 
to $\vartheta_{<\kappa}(\varphi)$.
\end{cor}

\begin{proof}
Immediate from Proposition~\ref{p: EA}. 
\end{proof}

\hide{
\DS{[It seems, we do not have, 
for a~given~$\varphi$, some fixed~$\nu<\kappa$ 
such that $\vartheta(\varphi)$ is equivalent 
to $\vartheta_{\le\nu\gen}(\varphi)$. Check!]}
}

%\newpage

%\subsection*{Monadic-like formulas}

A~signature~$\tau$ (and a~language $\mathcal L(\tau)$ 
in this signature) is called {\it monadic} iff 
any of its symbols has the arity~$\le1$, and 
{\it purely monadic} iff it is monadic and does not 
contain function symbols. %% other than constants??

A~standard observation, stated for the first time 
by L{\"o}wenheim~\cite{Lowenheim 1915}, is that 
all formulas of a~purely monadic language 
$\mathcal L_{\omega,\omega}(\tau)$ with equality 
are equivalent to Boolean combinations of existential 
(or universal) formulas 
(see, e.g., \cite{Borger et al}, Theorem~6.2.1; 
in fact, for $\tau$ consisting of $m$~unary predicate 
symbols, any satisfiable formula of the class 
$\Sigma^{0}_n\cup\Pi^{0}_n$ has a~model 
of size~$\le 2^m\cdot n$). 
We provide a~variant of this observation 
for arbitrary infinitary languages 
$\mathcal L_{\kappa,\lambda}$ and 
purely monadic languages without equality.

Given an arbitrary first-order (i.e., using only 
individual variables) language~$\mathcal L$, 
a~signature~$\tau$, and a~formula $\varphi$ of 
$\mathcal L(\tau)$, let us say that $\varphi$ is 
{\it monadic-like} iff each atomic subformula 
of~$\varphi$ contains at most one variable. E.g., 
so are open formulas in one parameter, 
or formulas in a~one variable fragment of 
a~first-order language, or else in its monadic 
fragment without equality. Note that formulas 
of a~language with equality, even if $\tau$ is 
empty, are generally not monadic-like. 
Below $\At\Sbf(\varphi)$ denote the set of 
atomic subformulas of a~given formula~$\varphi$. 

%\newpage 

\begin{tm}\label{t: monadic-like}
Let $\omega\le\lambda\le\kappa$, a~signature 
$\tau$, and a~monadic-like formula~$\varphi$ 
of $\mathcal L_{\kappa,\lambda}(\tau)$
be such that one of (i)--(iii) holds: 
\begin{enumerate}
\setlength\itemsep{-0.2em}
\item[(i)]
$\kappa=\lambda$ is strongly inaccessible;  
\item[(ii)] 
$\varphi$~is closed and 
$2^{2^{|\At\Sbf(\varphi)|}}<\kappa$; 
\hide{
  $\varphi$~is closed and 
  $2^{2^{\mu}}<\kappa$ where 
  $\mu:=|\At\Sbf(\varphi)|$; 
}
\item[(iii)]
$\varphi$~is closed and 
$2^{2^{|\tau|}}<\kappa$. 
\hide{
  $\varphi$~is closed and 
  $2^{2^{\mu}}<\kappa$ where $\mu:=|\tau|$. 
}
\end{enumerate}
Then $\varphi$ is equivalent to a~Boolean 
combination of open and existential (or 
universal) formulas of the finite-quantifier 
language $\mathcal L_{\kappa,\omega}(\tau)$ 
each in at most one variable. 
\end{tm}

%\newpage 

\begin{proof}
(i). 
First observe the following elementary facts. 
If $\varphi(x)$ is a~formula in a~single 
variable~$x$, then the formula 
$\bigexists_{\beta<\nu}y_\beta\,\varphi(x)$ 
is equivalent either to $\exists x\,\varphi(x)$ 
(that holds if $x$~appears among the 
variables~$y_\beta$), or to $\varphi(x)$ itself 
(otherwise). 
Let now each formula 
$\varphi_\alpha(x_\alpha)$, $\alpha<\mu$, 
has at most one variable~$x_\alpha$. Clearly, 
$
\bigexists_{\beta<\nu}y_\beta
\bigvee_{\!\alpha<\mu}\varphi_\alpha(x_\alpha)
$
is equivalent to 
$\bigvee_{\!\alpha<\mu}\psi_\alpha$ 
where each $\psi_\alpha$ is either
$\exists x_\alpha\,\varphi_\alpha$ 
(if $x_\alpha=y_\beta$ for some~$\beta$), or 
$\varphi_\alpha$ (otherwise). And the formula
$
\bigexists_{\beta<\nu}y_\beta\,
\bigwedge_{\alpha<\mu}\varphi_\alpha(x_\alpha)
$
is equivalent to the formula 
$$
\bigwedge_{\beta<\nu}
\Bigl(
\exists\, y_\beta 
\bigwedge_{\substack{\alpha<\mu\\
x_\alpha=y_\beta}}
\!\varphi_\alpha(x_\alpha)
\Bigr)
\wedge\!\!\!\! 
\bigwedge_{\substack{\alpha<\mu\\
x_\alpha\notin\{y_\beta\}_{\beta<\nu}}} 
\!\!\!\!\varphi_\alpha(x_\alpha). 
$$
And the dual statements hold for 
infinitary universal quantifiers.

Further, let us say that a~monadic-like 
formula $\varphi$ of 
$\mathcal L_{\kappa,\lambda}(\tau)$ has the 
{\it normal form} iff it is the disjunction
$\bigvee_{\alpha<\mu}\varphi_{\alpha}$
for some $\mu<\kappa$, where each 
$\varphi_\alpha$ is in at most one 
variable~$x_\alpha$ and has 
one of forms (a)--(c):
\begin{enumerate}
\setlength\itemsep{-0.2em}
\item[(a)] 
$
\bigwedge_{\beta<\nu_\alpha}
\psi_{\alpha,\beta}(x_\alpha)
$ 
for some $\nu_\alpha<\kappa$, where each 
$\psi_{\alpha,\beta}$ is either a~formula 
in $\At\Sbf(\varphi)$ or its negation;
\item[(b)] 
$\exists x_\alpha\;\psi_\alpha$ 
where $\psi_\alpha$ is a~formula in~(a);
\item[(c)] 
$\neg\,\exists x_\alpha\;\psi_\alpha$
where $\psi_\alpha$ is a~formula in~(a). 
\end{enumerate} 
Since any formula in the normal form is 
a~Boolean combination of open and existential 
formulas of $\mathcal L_{\kappa,\omega}(\tau)$,  
each in at most one variable, it suffices 
to show that any monadic-like formula is 
equivalent to a~formula in the normal form. 
We prove this by induction on construction 
of~$\varphi$. Essentially we need to show that 
whenever a~formula $\varphi$ is in the normal form 
then $\bigexists_{\beta<\nu}y_\beta\,\varphi$ 
can be represented in the normal form as well. 
Arguing by induction on the length of the Boolean 
combination and using the observation above, we push 
quantifiers inward the formula, replacing them and 
the connectives by the dual quantifiers and 
connectives when needed, and result in a~formula 
in the normal form, as required. Note that the 
arities of connectives generally increase at each 
step when we push quantifiers, but the condition 
$\kappa=\lambda$ is strongly compact guarantees 
that they remain~$<\kappa$. This proves~(i).

%\vskip+0.5em

(ii). 
Let now $\varphi$ be a~sentence in 
$\mathcal L_{\kappa,\lambda}(\tau)$ and 
let $\kappa'$ be a~strongly inaccessible 
cardinal~$\ge\kappa$. Then by~(i), $\varphi$ 
is equivalent to a~sentence $\varphi'$ 
in $\mathcal L_{\kappa',\kappa'}(\tau)$ 
having the normal form, more precisely, 
the form $\bigvee_{\alpha<\mu}\varphi_{\alpha}$
for some $\mu<\kappa'$, where each sentence
$\varphi_\alpha$ has the form described in (b) 
or~(c) above. (Note that (a)~is excluded since 
$\varphi$ is assumed to be closed. Note also that 
we do not need to assume the existence of 
an inaccessible cardinal since $\varphi'$ 
belongs to some smaller language.) 

Let $\eta:=|\At\Sbf(\varphi)|$. If $\eta\ge\omega$, 
then there are~$\le2^{\eta}$~formulas of the form 
described in (b) and~(c), so we have 
$\nu_\alpha\le2^{\eta}$. Furthermore, 
there are~$\le2^{2^{\eta}}$~their disjunctions, 
so we have $\mu\le2^{2^{\eta}}$. If $\eta<\omega$, 
all the $\nu_\alpha$ and $\mu$ are finite. 
Therefore, since $2^{2^{\eta}}<\kappa$, 
in fact $\varphi'$ is equivalent to a~sentence in 
$\mathcal L_{\kappa,\omega}(\tau)$, as required. 

%\vskip+0.5em

(iii). 
Note that any monadic-like formula $\varphi$ of 
$\mathcal L_{\kappa,\lambda}(\tau)$ is equivalent 
to such a~formula in which all bounded variables 
coincide, and that if this latter condition holds, 
then $|\At\Sbf(\varphi)|\le\max(|\tau|,\omega)$. 
Therefore, the assertion follows from~(ii). 
\end{proof}

%\newpage

\begin{tm}\label{t: monadic signature}
Let $\omega\le\lambda\le\kappa$ and let 
$\tau$~be a~purely monadic signature without 
equality such that one of (i),\,(ii) holds: 
\begin{enumerate}
\setlength\itemsep{-0.2em}
\item[(i)]
$\kappa=\lambda$ is strongly inaccessible;  
\item[(ii)] 
$2^{2^{|\tau|}}<\kappa$. 
\end{enumerate}
Then, up to the semantic equivalence, 
$\mathcal L_{\kappa,\lambda}(\tau)$ coincides 
with $\mathcal L_{\kappa,\omega}(\tau)$ 
(for~(ii), on sentences)
and is closed under~$\vartheta$. 
\end{tm}

\begin{proof}
By Theorem~\ref{t: monadic-like}, any monadic-like 
sentence~$\varphi$ is equivalent to a~sentence 
which is a~Boolean combination of existential 
sentences of $\mathcal L_{\kappa,\omega}(\tau)$, 
each in one variable. Furthermore, pushing the 
negation inward, we can assume that it is 
a~combination of conjunctions and disjunctions 
of one-variable existential and universal 
sentences of $\mathcal L_{\kappa,\omega}(\tau)$. 
By Corollary~\ref{c: EA} (and since 
$|\fnc(\tau)|=0$), for any model $\mathfrak A$ 
of~$\tau$, there is $\nu<\kappa$, the same 
for all models, such that 
in~$\mathfrak A$, $\vartheta(\varphi)$ and 
$\vartheta_{\le\nu}(\varphi)$ are equivalent. 
Furthermore, by Lemma~\ref{l: relativization}, 
$\vartheta_{\le\nu}(\varphi)$ is equivalent to 
a~sentence~$\chi'$ 
in $\mathcal L_{\kappa,\kappa}(\tau')$ where 
$\tau':=\tau\cup\{=\}$. Let us show that, in 
turn, $\chi'$ is equivalent to a~sentence~$\chi$ 
of $\mathcal L_{\kappa,\omega}(\tau)$. 

Recall that, by the proof of 
Lemma~\ref{l: relativization} 
(and taking into account $|\fnc(\tau)|=0$ again), 
the sentence~$\chi'$ is the formula 
$
\bigexists_{\alpha<\nu}x_\alpha\,
\varphi^{\{x_\alpha\}_{\alpha<\nu}}.
$
Since $\varphi$ can be assumed to be a~Boolean 
combination of one-variable existential sentences 
of $\mathcal L_{\kappa,\omega}(\tau)$, say 
$\exists y_\eta\,\psi_\eta(y_\eta)$ 
with open~$\psi_\eta$, $\eta<\xi$, 
the formula $\varphi^{\{x_\alpha\}_{\alpha<\nu}}$ 
is equivalent to the same Boolean combination of 
the formulas 
$
(\exists y_\eta\,
\psi_\eta(y_\eta))^{\{x_\alpha\}_{\alpha<\nu}},
$ 
$\eta<\xi$. 
Observe the following fact:

\begin{lm}\label{l: relativization of 1-parameter}
If $\psi$ is an open formula (of arbitrary~$\tau$) 
in one parameter, then 
\begin{align*}
\bigl(\exists y\,
\psi(y)\bigr)^{\{x_\alpha\}_{\alpha<\nu}}
\text{ iff }
\bigvee_{\alpha<\nu}
\psi(x_\alpha);
\\
\bigl(\forall y\,
\psi(y)\bigr)^{\{x_\alpha\}_{\alpha<\nu}}
\text{ iff }
\bigwedge_{\alpha<\nu}
\psi(x_\alpha). 
\end{align*}
\end{lm}

\begin{proof}
A~routine verification.   
\end{proof}

By this observation, the formula 
$\varphi^{\{x_\alpha\}_{\alpha<\nu}}$ is equivalent 
to a~Boolean combination of the formulas 
$\bigvee_{\!\alpha<\nu}\psi_\eta(x_\alpha)$,
$\eta<\xi$, and so, to an open formula, 
say~$\psi$, in $\mathcal L_{\kappa,\omega}(\tau)$. 
Thus, we showed that $\vartheta_{\le\nu}(\varphi)$ 
is expressed by the sentence 
$\bigexists_{\alpha<\nu}x_\alpha\,\psi$. 
The latter sentence is in 
$\mathcal L_{\kappa,\kappa}(\tau)$, 
but since it is obviously monadic-like, 
by Theorem~\ref{t: monadic-like}, 
it is equivalent to a~sentence~$\chi$ in 
$\mathcal L_{\kappa,\omega}(\tau)$, as required. 
\end{proof}

\hide{
\begin{rmk}
Obviously, if $\tau$ contains the equality 
symbol, then $\mathcal L_{\kappa,\lambda}(\tau)$ 
with $\lambda>\omega$ is not equivalent to 
$\mathcal L_{\kappa,\omega}(\tau)$; e.g., 
a~sentence (in $\tau:=\{=\}$) expressing 
countability is in 
$\mathcal L_{\kappa,\lambda}(\tau)$ but 
not in $\mathcal L_{\kappa,\omega}(\tau)$. 
\DS{[remove this trivial remark from 
the finalized text?]}
\end{rmk}
}

%\newpage

\section*{Satisfiability in extensions}

In this final section, we consider the dual 
operator~$\vartheta_{K,R}$ in which $R$~is the 
relation of extension of models (and $K$~is 
the class of models of a~given~$\tau$). As was 
said, we denote it just by~$\vartheta^*$; 
thus, $\mathfrak A\vDash\vartheta^*(\varphi)$
means that $\varphi$~is satisfied in some 
extension $\mathfrak B$ of~$\mathfrak A$.
Two main results of this section are as follows. 
First, for some~$\tau$, there are sentences 
$\varphi$ of $\mathcal L_{\omega,\omega}(\tau)$ 
with $\vartheta^*(\varphi)$ inexpressible 
by a~single sentence of the same language. 
Second, for all~$\tau$ and $\varphi$ in 
$\mathcal L_{\kappa,\lambda}(\tau)$ where 
$\kappa$ is $\omega$ or strongly compact, 
$\vartheta^*(\varphi)$ is expressed by 
a~universal theory in the same language, 
and so, by a~single universal sentence in 
$\mathcal L_{\infty,\lambda}(\tau)$. 
The latter shows that the expressibility of 
the extension modality (i.e.,~$\vartheta^*$) 
is generally easier than that of the submodel 
modality (i.e.,~$\vartheta$). Our final remark 
is that this expressibility result remains true 
for the second-order languages 
$\mathcal L^{2}_{\kappa,\lambda}$ 
where $\kappa$ is extendible.

We start from the following observation. 
Clearly, for every sentence~$\varphi$ of 
any language, $\vartheta^*(\varphi)$ is 
preserved under extensions of models, 
whence we get:

\begin{prp}\label{p: theta* universal} 
Let $\tau$~a~signature and $\varphi$ a~sentence 
in $\mathcal L_{\kappa,\lambda}(\tau)$ 
where either $\kappa=\omega$, or 
$\kappa=\lambda$ is strongly compact, or 
else $\kappa=\omega_1$ and $\lambda=\omega$.
The following are equivalent:
\begin{itemize}
\setlength\itemsep{-0.2em}
\item[(i)]
$\vartheta^*(\varphi)$ is equivalent to 
a~sentence in $\mathcal L_{\kappa,\lambda}(\tau)$;
\item[(ii)]
$\vartheta^*(\varphi)$ is equivalent to 
a~universal sentence in 
$\mathcal L_{\kappa,\lambda}(\tau)$. 
\end{itemize}
\end{prp}

\begin{proof}
Use Lemma~\ref{l: E and A}.
\end{proof}

%\newpage 

\begin{tm}\label{t: inexpressible theta*}
Let a~signature~$\tau$ contain at least one function 
symbol of arity~$\ge2$. Then there is a~$\Pi^{0}_2$ 
(i.e., universal-existential) sentence $\varphi$ in 
$\mathcal L_{\omega,\omega}(\tau)$ such that 
$\vartheta^*(\varphi)$ is not equivalent to 
any sentence of $\mathcal L_{\omega,\omega}(\tau)$. 
\end{tm}

\begin{proof} 
Let us first consider the partial case when 
the signature~$\tau$ consists of a~single binary 
function symbol, say $\tau:=\{\,\cdot\,\}$. 
Let $\varphi$ be the $\Pi^{0}_2$-sentence of 
$\mathcal L_{\omega,\omega}(\tau)$ expressing 
that the model has a~group structure; i.e., 
$\varphi$ is the conjunction of three sentences 
asserting the associativity of the operation 
$\,\cdot\,$ and the solvability of the equations 
$x\cdot y=z$ and $y\cdot x=z$ w.r.t.~$x$, for 
all $y$ and~$z$. Then $\vartheta^*(\varphi)$ 
expresses that the model of~$\tau$, i.e., 
a~groupoid, can be embedded into a~group. 

By Maltsev's classical result, the embeddability 
of a~semigroup into a~group is axiomatizable by 
a~countable set $T$ of quasi-identities but not 
by any its finite subset 
(\cite{Maltsev 1939,Maltsev 1940}, 
see also \cite{Bokut,Hollings}). 
It follows that it is not finitely axiomatizable 
at all (otherwise it would be axiomatizable by 
a~single~$\psi$, so $\psi$ would follow from~$T$, 
and so, from an its finite subset). Hence, 
the embeddability of a~groupoid into a~group is 
also not finitely axiomatizable, thus showing 
that $\vartheta^*(\varphi)$ is not equivalent 
to any single sentence in 
$\mathcal L_{\omega,\omega}(\tau)$, as required. 

Now we handle the general case by reducing it 
the considered partial one. Given a~$\tau$ with  
a~function symbol of arity~$\ge2$, say~$F$, 
let us consider the expansion 
$\tau':=\tau\cup\{\,\cdot\,\}$ where 
$\,\cdot\,$ is a~binary function symbol 
not belonging to~$\tau$. Note that 
the $\Pi^{0}_2$-sentence $\varphi$ from 
the previous part of the proof (i.e., expressing 
that $\,\cdot\,$~is a~group operation) is 
in $\mathcal L_{\omega,\omega}(\tau')$. 
Replacing all occurrences of the terms of the 
form $x\cdot y$ by the terms $F(x,y,\ldots,y)$, 
we transform $\varphi$ to the sentence 
$\varphi_F$ of $\mathcal L_{\omega,\omega}(\tau)$. 
We claim that $\vartheta^\ast(\varphi_F)$ is not 
equivalent to any single sentence of the same 
language $\mathcal{L}_{\omega,\omega}(\tau)$.

\par 
Toward a~contradiction, assume a~sentence 
$\psi$ of $\mathcal{L}_{\omega,\omega}(\tau)$ 
expresses $\vartheta^\ast(\varphi)$. 
Let $\chi$ be the sentence
$\forall x\forall y\;F(x,y,\ldots,y)=x\cdot y$,  
let $K$ be the class of all $\tau'$-models 
$\mathfrak A$ such that $\mathfrak A\vDash\chi$ 
and $\mathfrak A$ embeds into some $\tau'$-model 
$\mathfrak B$ with $\mathfrak B\vDash\varphi'$, 
and let $K^-$ be the class of 
$\{\,\cdot\,\}$-reducts of models in $K$.
Observe the following facts: 
\begin{enumerate}
\setlength\itemsep{-0.2em}
\item[(i)] 
$K^-$ is exactly the class of groupoids 
that are embeddable into groups; 
\item[(ii)] 
$K$ is axiomatizable by the sentence 
$\psi\wedge\chi;$
\item[(iii)] 
$K$ is axiomatizable by the theory 
$\varPsi\cup\{\chi\}$ where $\varPsi$ 
is the set of all universal 
$
\mathcal L_{\omega,\omega}(\{\,\cdot\,\})
$-consequences of the sentence~$\varphi_F$. 
\end{enumerate}
Here (i) and (ii) are immediate; for~(iii), 
recall the general fact of model theory: 
given a~theory~$T$, the class of all submodels 
of models of~$T$ is axiomatizable by the set 
of all universal consequences of~$T$ (see, e.g., 
\cite{Keisler}, Theorem~3.11; we extend this 
fact to compact infinitary languages in 
Theorem~\ref{t: expressing theta*} below). 
Applying this to $T:=\{\psi\wedge\chi\}$ in~(ii), 
we get~(iii). Further, we deduce from~(iii) 
by compactness that $K$ is axiomatizable by 
the sentence $\psi_0\wedge\chi$ where 
$\psi_0$~is $\bigwedge\varPsi_0$ for some 
finite $\varPsi_0\subseteq\varPsi$. Now it 
easily follows that $K^-$ is axiomatizable 
by the formula~$\psi_0$, which is, however, 
impossible as has already been proven. 
\end{proof} 

%\newpage 

\begin{ex}\label{e: expressible theta*} 
Let $\tau$ be as in 
Theorem~\ref{t: inexpressible theta*}. Then  
small modifications of this $\varphi$ can lead to 
$\vartheta^*(\varphi)$ expressible by a~sentence 
in $\mathcal L_{\omega,\omega}(\tau)$: 
\begin{enumerate}
\setlength\itemsep{-0.2em}
\item[(i)] 
if $\varphi$ be the $\Pi^{0}_2$-sentence defining 
a~quasigroup structure, i.e., the conjunction of 
cancellability and invertibility, then 
$\vartheta^*(\varphi)$ is equivalent to 
the cancellability (see \cite{Cohn});
\item[(ii)] 
if $\varphi$ be the $\Pi^{0}_2$-sentence defining 
an Abelian group structure, i.e., the conjunction of 
associativity, commutativity, cancellability, and 
invertibility, then $\vartheta^*(\varphi)$ is equivalent 
to associativity, commutativity, and cancellability. 
\hide{reference??}
\end{enumerate}
\hide{
\DS{[NB: in both cases, $\vartheta^*(\varphi)$ is 
equivalent to a~conjunctive member of~$\varphi$]}. 
}
\end{ex}

%\newpage 

Clearly, $\vartheta^*(\varphi)$ is expressible 
by a~theory (by a~sentence) iff the class of 
submodels of models of~$\varphi$ is axiomatizable 
(axiomatizable by a~single sentence). Let us show 
that the first case holds always whenever 
the language is compact. Moreover, if this is 
the case, then the second case holds iff there is 
a~strongest among all universal consequences 
of~$\varphi$. All this remains true on the class 
of models of a~given theory~$T$, and we present  
the proof for this slightly more general situation.

\begin{tm}\label{t: expressing theta*}
Let $\kappa=\omega$ or $\kappa$~a~strongly 
compact cardinal, let $\tau$ be an arbitrary 
signature, and let $T$ be a~theory and 
$\varphi$ a~sentence in 
$\mathcal L_{\kappa,\kappa}(\tau)$. Then: 
\begin{enumerate}
\setlength\itemsep{-0.2em}
\item[(i)] 
$\vartheta^*(\varphi)$ is always expressible 
by a~theory $\varTheta$ in 
$\mathcal L_{\kappa,\kappa}(\tau)$, 
namely, by 
$$
\varTheta:=
\bigl\{
\psi\in\mathcal L_{\kappa,\kappa}(\tau):
\psi\text{~is universal and }
T\cup\{\varphi\}\vDash\psi
\bigr\};
$$
\item[(ii)] 
$\vartheta^*(\varphi)$ is always expressible 
by a~single sentence in 
$\mathcal L_{\mu^+,\kappa}(\tau)$ 
where $\mu:=\max(\kappa,|\tau|)$; 
\item[(iii)] 
$\vartheta^*(\varphi)$ is expressible by a~single 
sentence in $\mathcal L_{\kappa,\kappa}(\tau)$ 
iff $\varTheta$ has a~formula that is strongest 
modulo~$T$ (i.e., there is some $\theta\in\varTheta$ 
such that $T\cup\{\theta\}\vDash\psi$ for every 
$\psi\in\varTheta$).  
\end{enumerate}
\end{tm}

\begin{proof}
Clearly, (i) implies (ii) (pick the sentence 
$\bigwedge\varTheta$) as well as (iii), so it 
remains to show~(i), i.e., that on models of~$T$, 
$\vartheta^*(\varphi)$ iff $\varTheta$. 

(Only if).
Suppose 
$\mathfrak A\vDash T\cup\{\vartheta^*(\varphi)\}$, 
so there is $\mathfrak B$ extending $\mathfrak A$ 
and such that $\mathfrak B\vDash T\cup\{\varphi\}$. 
Since universal formulas are preserved under 
submodels, we have $\mathfrak A\vDash\psi$ whenever 
$\psi$~is universal and $T\cup\{\varphi\}\vDash\psi$, 
hence $\mathfrak A\vDash\varTheta$. 

(If).
Suppose $\mathfrak A\vDash T\cup\varTheta$, and let 
$D_{\mathfrak A}$ be the diagram of~$\mathfrak A$
(in the signature $\tau':=\tau\cup\{c_a:a\in A\})$. 
Let us verify that the theory 
$D_{\mathfrak A}\cup T\cup\{\varphi\}$ has a~model. 

Toward a~contradiction, assume that 
$D_{\mathfrak A}\cup T\cup\{\varphi\}$ 
is not satisfiable. Since $\kappa=\omega$ 
or $\kappa$~strongly compact, and so, 
$\mathcal L_{\kappa,\lambda}$ satisfies 
the strong compactness theorem, there exists 
$\varGamma\in[D_{\mathfrak A}]^{<\kappa}$ 
such that $\varGamma\cup T\cup\{\varphi\}$ 
is not satisfiable. Therefore, 
$T\cup\{\varphi\}\vDash\psi$ where $\psi$~is 
the universal closure of the disjunction  
$\bigvee_{\!\chi\in\varGamma}\neg\,\chi'$ and 
$\chi'$~is obtained from~$\chi$ by replacing 
all constant symbols $c_a$, $a\in A$, by 
distinct variables, i.e., $\chi'$~has the form 
$\chi[x_a/c_a]$ where $a\mapsto x_a$ is 
a~one-to-one map of $A$ into the set of 
variables. (Note that, since $\varGamma$ can 
be of any cardinality~$<\kappa$, the universal 
quantifier preceding the disjunction can likewise 
be of that same size; therefore, we must assume 
$\kappa=\lambda$.) Since $\psi$ is universal, 
$\psi\in\varTheta$. Therefore, 
$\mathfrak A\vDash\psi$. However, this is 
impossible since $\mathfrak A\vDash\varGamma$. 
This shows that the theory 
$D_{\mathfrak A}\cup T\cup\{\varphi\}$ 
is satisfiable.
 
So pick any $\tau'$-model $\mathfrak B$ of 
$D_{\mathfrak A}\cup T\cup\{\varphi\}.$
As its $\tau$-reduct includes a~submodel 
isomorphic to~$\mathfrak A$, this shows 
$\mathfrak A\vDash\vartheta^{*}(\varphi)$, 
as required. 
\end{proof}

\hide{
It seems, there is a~known general result, check:

A~model~$\mathfrak A$ has an extension satisfying 
a~theory~$T$ (in the same vocabulary) if and only if 
$\mathfrak A$ satisfies all the universal sentences 
that are provable from~$T$. (Actually, there is an 
even more general version of the result; see the 
``Model Extension Theorem'' in Section~5.2 of 
Shoenfield's ``Mathematical Logic'' or Lemma~3.5.10(i) 
of Hinman's ``Fundamentals of Mathematical Logic''.)
}

\begin{rmk}\label{r: extendible} 
In~(iii), ``strongest modulo~$T$\,'' cannot 
be replaced by just ``strongest'' (under 
ordinary logical implication). E.g., let 
$\kappa=\omega$, $\tau:=\{P_n\}_{n<\omega}$, 
$T:=\{\forall x\,P_n(x):n<\omega\}$, and 
$\varphi:=\top.$ 
On models of~$T$, $\vartheta^*(\top)$ is valid and 
is therefore expressed by the single sentence~$\top$.  
Nevertheless, the set $\varTheta$ of universal 
consequences of $T\cup\{\top\}$ has no strongest 
element; any first-order candidate mentions only 
finitely many predicates $P_n$ and so cannot imply 
$\forall x\,P_m(x)$ for an unmentioned~$P_m$.

Clearly, Theorem~\ref{t: expressing theta*}, 
due to its general form, characterizes not 
only $\vartheta^*$ but also $\vartheta^{*}_K$ 
where $K$~is the class of models of a~given 
theory~$T$. Moreover, an analogous theorem 
holds for second-order languages 
$\mathcal L^{2}_{\kappa,\kappa}$ 
whenever $\kappa$~is an extendible cardinal, 
by virtue of Magidor's characterization of 
such~$\kappa$ via the second-order compactness 
theorem given in \cite{Magidor 1971}; see also 
\cite{Kanamori,Kunen 2003}.
\end{rmk}

%\newpage

%}

\vskip+0.6em

\begin{footnotesize} 
\noindent
{\sc HSE University}
\\
{\it E-mail address:} 
niknikols0@gmail.com
\end{footnotesize}

\begin{footnotesize} 
\noindent
{\sc Higher School of Modern Mathematics MIPT}
\\
{\it E-mail address:} 
d.i.saveliev@gmail.com 
\end{footnotesize}

\end{document}